\documentclass[a4paper,11pt,reqno]{amsart}
\usepackage{amsmath}
\usepackage[T1]{fontenc}
\usepackage{amssymb}
\usepackage{amsthm}
\usepackage{color}

\newcommand{\Ab}{\mathbf A}

\newcommand{\R}{\mathbb R}

\DeclareMathOperator{\tr}{tr}
\DeclareMathOperator{\dist}{dist}
\DeclareMathOperator{\supp}{supp}
\DeclareMathOperator{\Spec}{Spec}

\newtheorem{thm}{Theorem}[section]

\newtheorem{lem}[thm]{Lemma}

\newtheorem{proposition}[thm]{Proposition}

\newtheorem{rem}[thm]{Remark}

\numberwithin{equation}{section}

\title[Energy of bound states]{On the energy of bound states for magnetic Schr\"odinger operators}

\author[S. Fournais]{S\o ren Fournais}
\address[S. Fournais and A. Kachmar]{Department of Mathematical Sciences, University
  of Aarhus, Ny Munkegade, Building
  1530, DK-8000 \AA rhus C, Denmark}
\email[S. Fournais]{fournais@imf.au.dk}
\email[A. Kachmar]{ayman.kachmar@math.u-psud.fr}

\author[A. Kachmar]{Ayman Kachmar}

\date{\today}

\begin{document}

\begin{abstract}
We provide a leading order semiclassical asymptotics of
the energy of bound states for magnetic Neumann Schr\"odinger
operators in two dimensional (exterior)
domains with smooth boundaries.
The asymptotics is valid all the way up to the bottom of the essential
spectrum.
When the spectral parameter is varied near the value where bound
states become allowed in the interior of the domain, we show that the energy has a boundary and a bulk component.
The estimates rely on coherent states, in particular on the
construction of `boundary coherent states', and magnetic Lieb-Thirring
estimates.
\end{abstract}

\maketitle

\section{Introduction}

Let $\Omega'\subset\mathbb R^2$ be an open and bounded domain with
regular boundary. 
We will consider both the case of
interior domains $\Omega = \Omega'$ and exterior domains $\Omega
=\mathbb R^2\setminus\Omega'$.

Consider the magnetic Schr\"{o}dinger operator, or magnetic
Laplacian, in $\Omega$. It has been observed by many authors (see
for example \cite{HeMo}) that the presence of Neumann boundary
conditions has an effect similar to that of a negative electric
potential
(also Robin boundary conditions have a similar effect,
see \cite{Kach1, Kach6}). For the case of the present discussion, consider a
constant magnetic field, more generally we will impose the
hypothesis \eqref{hyp-B} below.
One aspect of the analogy is that the Neumann  boundary condition
leads to a discrete  spectrum below the lowest Landau level (in the
case of an exterior domain, the (magnetic) Laplacian may have
essential spectrum). It is that spectrum which we will discuss in
the present paper.

One motivation for our work is the analysis of type
II superconductivity for applied magnetic fields close to the second
critical field (see \cite[Problem~2.2.8, p.~491]{Pa} for a
discussion of this point). In the setting of superconductivity, one
encounters the asymptotic regime of a large magnetic field, but this
is equivalent  (through a simple change of parameter) to  the
semiclassical regime considered in this paper.

Many different investigations of the discrete spectrum close to the
lowest Landau level have appeared (see \cite{MR, mP, RS} and references
therein).
Of particular importance for our present investigation is
the recent paper \cite{Fr}. In that paper the counting function is
studied up to an energy strictly below the lowest Landau level.
Inspired by work on the magnetic problems without boundaries---but
with electric potential---see \cite{LSY}, we shift focus from the
counting function to the energy, i.e. the sum of the eigenvalues.
This allows us to obtain 'semiclassical' results all the way up to
the lowest Landau level (Theorem~\ref{FoKa-thm1} below).
Furthermore, by varying slightly the energy up to which we sum the
eigenvalues, we can demonstrate how the bulk or interior bound
states start to contribute to the leading order asymptotics for the
energy (Theorem~\ref{thm2-FK} below).

Let us now move to the precise statement of our results. We consider
the (Neumann) Schr\"odinger operator with magnetic field~:
\begin{equation}\label{JMP07Op}
P_{h,\Omega} =-(h\nabla-i\Ab)^2=(-ih\nabla - \Ab)^2\,,
\end{equation}
whose domain is,
\begin{eqnarray}\label{JMP07DOp}
D\left(P_{h,\Omega}\right)=\big{\{}u\in L^2(\Omega)&:&
(h\nabla-i\Ab)^ju\in
L^2(\Omega),~j=1,2,\\
&&\nu\cdot(h\nabla-i\Ab)u=0~{\rm on}~\partial\Omega\big{\}}.\nonumber
\end{eqnarray}
Here $h>0$ is a small parameter (the so called semi-classical
parameter), $\nu$ is the unit outward normal vector of the boundary
$\partial\Omega$ and $\Ab\in C^2(\overline \Omega ;\mathbb R^2)$ is
a given vector field---the vector potential.
The magnetic field is a function and is given by
\begin{equation}\label{MF}
B={\rm curl}\,\Ab = \partial_{x_1} A_2 - \partial_{x_2} A_1\,.
\end{equation}
With this magnetic field we associate the quantities
\begin{equation}\label{b,b'}
b=\inf_{x\in\overline\Omega} B(x)\,,\quad
b':=\inf_{x\in\partial\Omega} B(x)\,.
\end{equation}
Assuming that $b>0$, we know that the following leading order
asymptotic expansion holds for the bottom of the spectrum of
$P_{h,\Omega}$ (see for example \cite{HeMo})
\begin{align}
  \label{eq:2}
  \inf \Spec P_{h,\Omega} = h \min ( b, \Theta_0 b') + o(h).
\end{align}
Here $\Theta_0\in\,]0,1[$ is a universal constant (the definition will
be recalled in
\eqref{Th-gam} below).

We shall assume that the magnetic field is bounded, positive  and satisfies
\begin{equation}\label{hyp-B}
b>\Theta_0b' > 0.
\end{equation}
Notice that the hypothesis \eqref{hyp-B} is satisfied in the case of
a constant magnetic field. Under the hypothesis \eqref{hyp-B}
eigenvalues strictly below $bh$ are associated with eigenfunctions
localized near the boundary. One of the objectives of the present
paper is to prove that---for the energy---this remains true all the
way up to the value $bh$.

In order to state our main results, we need to  recall some facts
concerning the harmonic oscillator on the semi-axis $\mathbb R_+$.
For $\xi\in\mathbb R$, we denote by $\mu_1(\xi)$ the lowest
eigenvalue of the operator
$$-\partial_t^2+(t-\xi)^2\quad{\rm in }\quad L^2(\mathbb R_+)$$
with standard Neumann boundary condition at the origin. It is well
known (see \cite{BoHe, DaHe}) that the function $\xi\mapsto\mu_1(\xi)$
is smooth,
\begin{align}
  \label{eq:1}
  \mu_1(\xi)< 1, \text{ for all }\xi\in\mathbb R_+,\qquad
\mu_1(\xi)> 1, \text{ for all }\xi\in\mathbb R_{-},
\end{align}
and  the integral
$$\int_0^\infty\left(\mu_1(\xi)-1\right)\,d  \xi = - \int_{\R} [
\mu_1(\xi) - 1 ]_{-} d  \xi
$$
is  negative and finite.
Here we introduced the notation $[x]_{-}$, more generally, we will use the following positive functions
\begin{align*}
[x]_{+} = \begin{cases}
x, & x\geq 0, \\ 0, & x < 0,
\end{cases}
\qquad
[x]_{-} = \begin{cases}
0, & x\geq 0, \\ -x, & x < 0.
\end{cases}
\end{align*}
Our result is the following.

\begin{thm}\label{FoKa-thm1}
Suppose $\Omega$ is either an exterior or an interior domain.
Suppose $B$ satisfies \eqref{hyp-B}. Then the spectrum of $P_{h,\Omega}$
below $bh$ is discrete,
$$\sigma\left(P_{h,\Omega}\right)\cap
\,]-\infty,bh[\,=\{e_1(h),e_2(h),\hdots\}\,,$$
and the sequence $\{e_j(h)-bh\}_{j\geq1}$  is summable,
$$\sum_j\left[e_j(h)-bh\right]_-=-{\rm tr}\left((P_{h,\Omega}-bh)\mathbf
1_{]-\infty,bh[}(P_{h,\Omega})\right) \quad{\rm is~finite.}$$
Moreover, the following asymptotic formula holds,
\begin{multline}
\label{eq:EnergyResult}
\lim_{h\to0} h^{-1/2}
\sum_j\left[e_j(h)-bh\right]_-\\
=\frac1{2\pi} \int_{\partial\Omega}\int_{-\infty}^\infty
B(x)^{3/2}\Big[\frac{b}{B(x)}-\mu_1(\xi) \Big]_+\,d  \xi\,d
s(x)\,.\end{multline}
Here $d  s(x)$ denotes integration with
respect to arc-length along the boundary $\partial \Omega$.
\end{thm}

\begin{rem}\label{rem-bnonconstant}
{\rm Assuming that the magnetic field $B={\rm curl}\,\Ab$ is constant,
  $B(x)=b$, the asymptotic formula of Theorem~\ref{FoKa-thm1} reads
$$\lim_{h\to0} h^{-1/2}
\left(\sum_j\left[e_j(h)-bh\right]_-\right)=
\frac{|\partial\Omega|b^{3/2}}{2\pi}\int_0^\infty\left(1-\mu_1(\xi)\right)\,d
\xi\,.$$}
\end{rem}

\begin{rem}\label{rem-rupert}
{\rm In \cite{Fr}, an asymptotic formula is obtained for the number
$N(\lambda h)$ of
eigenvalues of $P_{h,\Omega}$ below $\lambda h$, for a given
$\lambda<b$. The precise result is the following~:
\begin{equation}\label{rem-Ru-eq1}
\lim_{h\to0} h^{1/2}
N(\lambda h)=\frac1{2\pi}\int_{\{(x,\xi)\in\partial\Omega\times\mathbb R~:~
B(x)\mu_1(\xi)<\lambda\}}B(x)^{1/2}\,d  \xi d  s(x)\,.\end{equation}
Integration of \eqref{rem-Ru-eq1} yields the following formula for the
energy,
\begin{multline}\label{rem-Ru-eq2}
\lim_{h\to0} h^{-1/2}
\left(\sum_j\left[e_j(h)-\lambda h\right]_-\right)\\
=\frac1{2\pi} \int_{\partial\Omega}\int_0^\infty
B(x)^{3/2}\left[\frac{\lambda}{B(x)}-\mu_1(\xi)\right]_+  d
\xi\, d  s(x)\,.
\end{multline} 
However, the proof we give
to Theorem~\ref{FoKa-thm1} gives equally (\ref{rem-Ru-eq2}) (we only
give the details for the harder case $\lambda = b$), and hence, by
differentiating \eqref{rem-Ru-eq2}, (more precisely, the
'differentiation' needed is the technique used to go from energies to
densities in semiclassical problems, see \cite{ELSS} for details) we provide an alternative proof
of \eqref{rem-Ru-eq1}.}
\end{rem}

In the next theorem, we restrict to the case of interior domains,
i.e. bounded $\Omega$. In this specific case, the operator $P_{h,\Omega}$
has compact resolvent and hence its spectrum is purely {\it
  discrete}. Let us denote by $\{e_j(h)\}$ the increasing sequence of eigenvalues
of $P_{h,\Omega}$ (counted with multiplicity).

\begin{thm}\label{thm2-FK}
Suppose $\Omega$ is bounded, smooth
and $B$ is constant in $\overline\Omega$.
Given $a\in\mathbb R$, the following limit holds,
\begin{multline}\label{eq-thm2-FK}
\lim_{h\to0}h^{-1/2}\sum_{j} \left[e_j(h)-bh-a h^{3/2}\right]_-=\\
\frac{|\partial\Omega|b^{3/2}}{2\pi}\int_{\R}(\mu_1(\xi)-1)_- d \xi+
\frac{|\Omega|b}{2\pi}[a]_+\,.\end{multline}
\end{thm}

\begin{rem}\label{rem-transition}
{\rm The first term on the r.h.s. of \eqref{eq-thm2-FK} corresponds
to {\it boundary states} (i.e. eigenfunctions localized near the
boundary) and the second one corresponds to {\it bulk states}.
Theorem~\ref{thm2-FK} sharpens our understanding of the transition
from {\it boundary states} to {\it bulk
  states}.}
\end{rem}

\begin{rem}{\rm
Since $\inf \Spec_{\rm ess} P_{h,\Omega} = bh$ in the case of exterior
domains, we see that the left hand side of \eqref{eq-thm2-FK} diverges
for any $a>0$. That forces us to restrict to bounded
$\Omega$. Our methods would also apply to non-constant $B$, but the
order to which the bulk term appears depends on the local behavior of
$B$ near the set $\{ x: B(x) = b\}$. For simplicity of exposition we
therefore restrict to the case of constant field.}
\end{rem}

The approach we follow is inspired by that of Lieb, Solovej and
Yngvason \cite{LSY}, but extended with the construction of coherent
states for the half-plane operator. In order to control
the errors resulting from the approximations, we make use of a
Lieb-Thirring inequality for magnetic operators (see \cite{ES, LSY}
and references therein), and a rough estimate of the energy of bound
states for the case of a cylindric domain.

The paper is organized as follows. We collect in
Section~\ref{sec-prelim} some preliminaries. In
Section~\ref{sec-torus}, we determine a rough bound for the energy
of bound states of the operator (\ref{JMP07Op}) in the case when the
domain $\Omega$ is a cylinder.
Section~\ref{sec-coherentstates} is devoted to the construction of
coherent states. In Section~\ref{sec-proof}, we prove
Theorem~\ref{FoKa-thm1}. Finally, in Section~\ref{sec-proof-thm2} we
give the additional details to achieve the proof of Theorem~\ref{thm2-FK}.

\section{Preliminaries}\label{sec-prelim}

\subsection{Lieb-Thirring inequality}

Let $B\in C^1(\mathbb R^2;\mathbb R)\cap L^\infty(\mathbb R^2)$ be a
magnetic field such that $B(x)>0$ for all $x\in\mathbb R^2$. The
vector field defined by
$$A(x)=\frac12\left(-\int_0^1 sB(sx)x_2\, d  s,
\int_0^1 sB(sx)x_1\, d  s\right)\quad\forall~x=(x_1,x_2)\in\mathbb
R^2\,,$$
provides a magnetic potential for $B={\rm curl}\,A$.

Consider
the Schr\"odinger (Pauli) operator
$$H_{\R^2}=-(\nabla-i A)^2-B\quad{\rm in}~L^2(\mathbb R^2).$$
We have the following Lieb-Thirring estimate for the negative
eigenvalues of $H_{\R^2}+V$ (see \cite{ES} and references therein).

\begin{thm}\label{thm-ES}
There exists a universal constant $C>0$ such that the following
estimate is valid for the sum of the negative eigenvalues
$\{e_j\}_{j\geq1}$ of the operator $H=H_{\R^2}+V$,
$$\sum_j|e_j|\leq C\left(\|B\|_{L^\infty(\mathbb R^2)}
\int_{\R^2}[V]_-\, d  x+ \int_{\R^2} [V]_-^2\, d  x\right)\,.$$
\end{thm}

\subsection{Variational principles}

Let $H$ be a self-adjoint operator in $L^2(\mathbb R^2)$ (of domain
$D(H)$) such that
$$({\rm H})\quad
\left\{
\begin{array}{l}
\inf\,\sigma_{\rm ess}(H)\geq 0\\
H\mathbf 1_{]-\infty,0[}(H) {\rm~is~trace~class}\,.
\end{array}
\right.$$

We shall need the following two simple variational principles
concerning the operator $H$, which are frequently used in \cite{LSY,
ES}.

\begin{lem}\label{lem-VP1}
Let $\gamma$ be a bounded operator such that $0\leq\gamma\leq 1$ (in
the sense of quadratic forms) and the operator $H\,\gamma$ is trace
class. Then it holds that,
$${\rm tr}\left(H\mathbf 1_{]-\infty,0[}(H)\right)\leq {\rm
  tr}(H\,\gamma)\,.
$$
\end{lem}

\begin{lem}\label{lem-VP2}
Assume that the operator $H$ satisfies the hypothesis {\rm (H)}.
Then it holds that,
$${\rm tr}\left(H\mathbf 1_{]-\infty,0[}(H)\right)
=\inf\sum_{j=1}^N\langle f_j\,,\, H \,f_j\rangle\,,$$ where the
infimum is taken over all orthonormal families
$\{f_1,f_2,\ldots,f_N\}\subset D(H)$ and $N\geq1$.
\end{lem}

\subsection{Boundary coordinates}
The closed quadratic form associated with the operator \eqref{JMP07Op} is
\begin{equation}\label{qf}
q_{h,\Omega}(u)=\int_\Omega|(h\nabla-i\Ab)u|^2\, d  x\,,
\end{equation}
with the form domain
\begin{equation}\label{qf-D}
H^1_{h,\Ab}(\Omega)=\{u\in L^2(\Omega)~:~(h\nabla-i\Ab)u\in
L^2(\Omega)\}\,.
\end{equation}
The following magnetic potential generates a constant unit
magnetic field,
\begin{equation}\label{A0}
\Ab_0(x_1,x_2)=(-x_2,0)\,,\quad \forall~(x_1,x_2)\in\mathbb
R\times\mathbb R\,.
\end{equation}
The quadratic form,
\begin{equation}\label{qf-model}
H^1_{h,b\Ab_0}(\mathbb R\times\mathbb R_+)
\ni u\mapsto\int_{\mathbb R\times\mathbb
R_+}|(h\nabla-i b\Ab_0)u|^2\, d  x\,,
\end{equation}
with $b>0$, will serve as a model form. Actually, when the function
$u\in H^1_{h,\Ab}(\Omega)$ is supported near the boundary, the form
(\ref{qf-model}) turns out to approximate (\ref{qf}). In order to
make this precise, we  introduce a convenient coordinate
transformation valid in a sufficiently thin tubular neighborhood of the
boundary\footnote{Here we assume for simplicity that
$\partial\Omega$ is connected. In general, $\partial\Omega$ has
finite connected components and therefore we should work on each
component independently.}. For more details on these coordinates, see for instance \cite[Appendix A]{HeMo}.
\begin{equation}\label{phi0}
\Phi_{t_0}:~\Omega(t_0)\ni x\mapsto (s(x),t(x))\in
\frac{|\partial\Omega|}{2\pi}\mathbb S^1\,\times\,]0,t_0[,\end{equation}
where for $t_0>0$, $\Omega(t_0)$ is the tubular neighborhood of
$\partial\Omega$:
$$\Omega(t_0)=\{x\in\Omega~:~{\rm dist}(x,\partial\Omega)<t_0\}.$$
Let us mention that $t(x)={\rm dist}(x,\partial\Omega)$ measures the
distance to the boundary and $s(x)$ measures the curvilinear
distance in $\partial\Omega$. We shall use the usual identification between the circle $\frac{|\partial\Omega|}{2\pi}\mathbb S$
and the interval $[0,|\partial\Omega|\,[$.
The Jacobian of the transformation $\Phi_{t_0}$ is equal to
$$1-tk(s)\,,$$
where $k$ denotes the curvature of $\partial\Omega$.

Since $t_0$ will be fixed once and for all, we will sometimes omit it
from the notation and simply write $\Phi$ instead of $\Phi_{t_0}$.

Using the coordinate transformation $\Phi_{t_0}$, we associate to
any function $u\in L^2(\Omega)$, a function $\widetilde u$ defined
in $ [0, |\partial\Omega|[\,\times\,[0,t_0]$ by,
\begin{equation}\label{VI-utilde}
\widetilde u(s,t)=u(\Phi_{t_0}^{-1}(s,t)).
\end{equation}
Furthermore, the function $\widetilde u$ extends naturally to a
$|\partial\Omega|$-periodic function in $s\in\mathbb R$.

We get then
the following change of variable formulae.
\begin{proposition}\label{App:transf}
Let $u\in H^1_\Ab(\Omega(t_0))$. We write $\widetilde
u(s,t)=u(\Phi_{t_0}(s,t))$,
\begin{align*}
\tilde \Ab_1 = \Ab_1 \circ \Phi_{t_0},\qquad \tilde \Ab_2 = \Ab_2
\circ \Phi_{t_0}\,.
\end{align*}
Then we have~:
\begin{multline}\label{App:qfstco}
\int_{\Omega(t_0)}\left|(\nabla-i\Ab)u\right|^2dx =
\int_{0}^{|\partial\Omega|} \int_{0}^{t_0}\left[
|(\partial_s-i\tilde \Ab_1)\widetilde
u|^2\right.\\\left.+(1-tk(s))^{-2}|(\partial_t-i\tilde
\Ab_2)\widetilde u|^2\right](1-tk(s))\,dsdt,
\end{multline}
and
\begin{equation}\label{App:nostco}
\int_{\Omega(t_0)} |u(x)|^2\,dx=\int_{0}^{|\partial\Omega|}
\int_{0}^{t_0} |\widetilde u(s,t)|^2(1-tk(s))\,dsdt.
\end{equation}
\end{proposition}

We will use the symbol $U_{\Phi}$ for the operator that maps $u$
to $\widetilde u$.
We shall frequently make use of the next standard lemma, taken from
\cite[Lemma~3.5]{Fr}.

\begin{lem}\label{VI-FrLem3.5}
There exists a constant $C>0$ and  for all
$S_1\in[0,|\partial\Omega|\,[$\,,
$S_2\in]S_1,|\partial\Omega|\,[$\,, there exists a function $\phi\in
C^2_0([S_1,S_2]\times[0,t_0];\mathbb R)$ such that, for all
$$\widetilde S\in[S_1,S_2],\quad \mathcal T \in
  ]0,t_0[,\quad
\varepsilon\in[C\mathcal T,Ct_0],$$ and for all $u\in H_{h,\Ab}^1(\Omega)$
satisfying
$${\rm supp}\,\widetilde u\subset [S_1,S_2]\times[0,\mathcal T],$$
one has the following estimate,
\begin{multline*}
\left|q_{h,\Omega}(u)-\int_{\mathbb R\times\mathbb R_+}
|(h\nabla-i\widetilde B\Ab_0)e^{i\phi/h}\widetilde u|^2\, d  s d  t\right| \\
 \leq\,\int_{\mathbb R\times\mathbb R_+} \left(\varepsilon
|(h\nabla-i\widetilde B\Ab_0)e^{i\phi}\widetilde u|^2
+C\varepsilon^{-1}\left((S^2+\mathcal T^2)^2+h^2\right) |\widetilde
u|^2\right)\, d  s d  t.
\end{multline*}
Here, $S=S_2-S_1$, $\widetilde B=\widetilde B(\widetilde S,0)$, the
function $\widetilde u$ is associated to $u$ by (\ref{VI-utilde})
and extended by $0$ on $\R\times\R_+\setminus{\rm supp}\,\widetilde
u$.
\end{lem}

\subsection{A family of one-dimensional differential operators}
\label{Sub-Sec-M0-1}
Let us recall the main results obtained in \cite{DaHe, HeMo}
concerning the
family of harmonic oscillators with Neumann boundary condition.
Given $\xi\in\mathbb R$, we
define the quadratic form,
\begin{equation}\label{k2-qfk1}
B^1(\mathbb R_+)\ni u\mapsto q[\xi](u)=\int_{\mathbb
R_+}|u'(t)|^2+|(t-\xi)u(t)|^2 d  t,
\end{equation}
where, for a positive integer $k\in\mathbb N$ and a given interval
$I\subseteq\mathbb R$, the space $B^k(I)$ is defined by~:
\begin{equation}\label{Bk-sp}
B^k(I)=\{u\in H^k(I);\quad t^ju(t)\in L^2(I),\quad \forall
j=1,\ldots, k\}.
\end{equation}
Since the quadratic form (\ref{k2-qfk1}) is closed and symmetric it
defines a unique self-adjoint operator  $\mathcal L[\xi]$. This
operator has domain,
$$D(\mathcal L[\xi])=\{u\in B^2(\mathbb R_+);\quad
u'(0)=0\},$$ and is the realization of the differential operator,
\begin{equation}\label{L-dop}
\mathcal L[\xi]=-\partial_t^2+(t-\xi)^2,
\end{equation}
on the given domain.
We denote by
$\{\mu_j(\xi)\}_{j=1}^{+\infty}$ the increasing sequence
of eigenvalues of $\mathcal L[\xi]$, which are all simple.
By the min-max principle, we have,
\begin{equation}\label{l1-ga,xi}
\mu_1(\xi)=\inf_{u\in B^1(\mathbb
R_+),u\not=0}\frac{q[\xi](u)}{\|u\|^2_{L^2(\mathbb R_+)}}.
\end{equation}
It follows from analytic perturbation theory (see \cite{Ka}) that the functions
$$\mathbb R\ni\xi\mapsto\mu_j(\xi)$$ are analytic.

As recalled in the introduction, $\mu_1(0)=1$ and
$|\mu_1(\xi)-1|$ decays like $\exp(-\xi^2)$ as $\xi\to+\infty$ (see
\cite{BoHe}) thus yielding that
$$\int_{0}^\infty(\mu_1(\xi)-1)\, d
\xi=-\int_{\R}[\mu_1(\xi)-1]_-\, d \xi\quad{\rm is~finite.}$$
We define the constant~:
\begin{equation}\label{Th-gam}
\Theta_0=\inf_{\xi\in\mathbb R}\mu_1(\xi).
\end{equation}
 Let us
recall an important consequence of standard Sturm-Liouville theory
(c.f.. \cite[Lemma~2.1]{Fr}).
\begin{lem}\label{RuFr-Lemm-kach}
The second eigenvalue satisfies,
$$\inf_{\xi\in\mathbb R}\mu_2(\xi)>1\,.$$
\end{lem}

Notice that part of this conclusion is a consequence of the analysis of
Dauge-Helffer \cite{DaHe}, who show that the infimum of $\mu_2(\xi)$
is attained for a unique $\xi_2\in\mathbb R$.

\section{Rough energy bound for the cylinder}\label{sec-torus}

Let us consider the operator (\ref{JMP07Op}) in the particular case
of a cylindric domain
$$\Omega=[0,S]\,\times\,]0,h^{1/2}T[\,.$$ Functions in the domain
of  $P_{h,\Omega}$ satisfy Neumann condition at $t=0$, periodic
conditions at $s\in\{0,S\}$ and Dirichlet condition at $t=h^{1/2}T$.
We assume in addition that the magnetic field is constant $B(x)=b$,
$b>0$, and that the magnetic potential $A=A_0$ is the one
given in (\ref{A0}).
In this particular case, the operator has compact resolvent, hence
the spectrum consists of an increasing sequence of eigenvalues
$(e_j)_{j\geq1}$ converging to $+\infty$. In particular, given
$\lambda>0$, the energy
\begin{equation}\label{energy-torus}
\mathcal
E(\lambda,b,S,T)=\sum_j\left[hb(1+\lambda)-e_j\right]_+\end{equation}
is finite. Our aim in this section is to provide a rough estimate
of this
energy.

\begin{lem}\label{roughestimate}
There exist positive constants $T_0$ and $\lambda_0$ such that, for
all $S>0$, $b>0$, $T\geq \sqrt{b}\,T_0$ and $\lambda\in]0,\lambda_0]$, we have,
$$\mathcal E(\lambda,b,S,T)\leq  (1+\lambda)hb\left(\frac{ST}{2\pi\sqrt{h}}+1\right)\,.$$
\end{lem}
\begin{proof}
By separation of variables and a scaling we
may decompose $P_{h,\Omega}$ as a direct sum:
$$\bigoplus_{n\in\mathbb Z}\,hb\left(-\frac{ d  ^2}{ d  t^2}+
(2\pi nh^{1/2}b^{-1/2}S^{-1}+t)^2\right)\quad {\rm in}~
\bigoplus_{n\in\mathbb Z} L^2(]0,T/\sqrt{b}[)\,,$$ with Neumann boundary
condition at the origin and Dirichlet condition at $t=T/\sqrt{b}$. Therefore,
we may express the energy (\ref{energy-torus}) in the form,
\begin{equation}\label{sep-variables}
\mathcal E(\lambda,b,S,T) =hb\sum_{\substack{n\in\mathbb
Z\\j\in\mathbb N}}\left[ 1+\lambda-\mu_j(2\pi
nh^{1/2}b^{-1/2}S^{-1};T/\sqrt{b})\right]_+\,.\end{equation} Here, for a
given $\xi\in\mathbb R$, we denote by $\mu_j(\xi;\mathcal T)$ the increasing
sequence of eigenvalues for the operator
$$-\partial_t^2+(t-\xi)^2\quad{\rm in}~L^2(]0,\mathcal T[)$$
with Neumann condition at the origin, and Dirichlet
condition at $t=\mathcal T$. Notice that for a fixed $\mathcal T$, the
min-max principle gives immediately, $\mu_1(\xi;\mathcal T)\to +\infty$ as
$\xi\to+\infty$. Hence, the sum on the right hand
side of (\ref{sep-variables}) is finite.
It follows also from the  min-max principle  that $\mu_2(\xi;\mathcal T)\geq
\mu_2(\xi)$, where $ \mu_2(\xi)$ is the second eigenvalue of the
operator (\ref{L-dop}). Now, Lemma~\ref{RuFr-Lemm-kach} gives the existence of a
sufficiently small $\lambda_0$ such that
$\mu_2(\xi)>1+\lambda_0$ for all $\xi\in\mathbb R$. Thus, taking
$\lambda\in]0,\lambda_0[$,
\begin{equation}\label{sep-variables1}
\mathcal E(\lambda,b,S,T)=hb\sum_{n\in\mathbb Z}\left[
1+\lambda-\mu_1(2\pi
nh^{1/2}b^{-1/2}S^{-1};T/\sqrt{b})\right]_+\,.\end{equation}
By
(\ref{sep-variables1}), if one can localize the set,
$$\{\xi\in\mathbb R~:~\mu_1(\xi;T/\sqrt{b})\leq 1+\lambda\},$$ then
one gets immediately an estimate of the energy $\mathcal
E(\lambda,b,S,T)$.
Notice that for $t\leq \mathcal T$ and $|\xi|\geq 2\mathcal T$, it holds that
$(t-\xi)^2\geq \mathcal T^2$\,, hence by the min-max principle,
$\mu_1(\xi;\mathcal T)\geq \mathcal T^2$\,. Therefore, choosing $\mathcal T_0^2> 1+\lambda_0$,
it holds for $\mathcal T\geq \mathcal T_0$ and $\lambda\in]0,\lambda_0[$\,,
$$\mu_1(\xi;\mathcal T)\leq 1+\lambda\implies
|\xi|\leq 2\mathcal T\,.$$ From the above localization, the estimate of
Lemma~\ref{roughestimate} becomes a consequence of
(\ref{sep-variables1}).
\end{proof}

\section{Generalized eigenprojectors}\label{sec-coherentstates}

\subsection{Eigenprojectors in $\mathbb
R^2$.}\label{subsec-cs-plane} Let $\Ab_0$ be the magnetic potential
given in (\ref{A0}), $b>0$ and $P_{h,b,\mathbb R^2}$ the
self-adjoint operator
$$P_{h,b,\mathbb R\times\mathbb R_+}=-(h\nabla-ib\Ab_0)^2\quad{\rm
  in}~L^2( \R^2)\,.$$
  We recall in this section the well-known
eigenprojectors for the operator $P_{h,b,\R^2}$. These are
projectors $\Pi_j^L(h,b)$, $j\geq1$,  on the Landau levels that
satisfy in particular,
$$P_{h,b,\R^2}\Pi_j^L(h,b)=(2j-1)\Pi_j^L(h,b)\,,\quad
\sum_{j=1}^\infty\Pi_j^L(h,b)={\rm Id}_{L^2(\R^2;\mathbb C)}\,.
$$
The integral kernels (denoted again by
$\Pi_j^L(h,b)$) are given explicitly (see \cite[(3.13)]{LSY})\,,
\begin{multline*}
\Pi_j^L(h,b)(x,y)= \frac{b}{2\pi
h}\exp\left(\frac{ib(y_1y_2-x_1x_2)}{2h}\right)\\
\times\exp\left(\frac{ib(x_1y_2-x_2y_1)}{2h}-\frac{b|x-y|^2}{4h}\right)
L_{j-1}\left(\frac{b|x-y|^2}{2h}\right)\,,
\end{multline*}
where $L_j$ are Laguerre polynomials normalized so that
$L_j(0)=1$. We will need the fact that,
\begin{align}
  \label{eq:3}
  \Pi_j^L(x,x)=\frac{b}{2\pi h}\,,\quad\forall~x\in \R^2\,.
\end{align}

\subsection{Eigenprojectors in the half-space}
We construct in this section projections on the (generalized)
eigenfunctions for the operator
(\ref{JMP07Op})
in the case $\Omega=\mathbb R\times\mathbb R_+$.

Let $\Ab_0$ be the magnetic potential given in (\ref{A0}), $b>0$ and
$P_{h,b,\mathbb R\times\mathbb R_+}$ the Neumann realization of the
operator
$$P_{h,b,\mathbb R\times\mathbb R_+}=-(h\nabla-ib\Ab_0)^2\quad{\rm
  in}~L^2(\mathbb R\times\mathbb R_+)\,.$$

Let us denote by $\left(u_j(\cdot;\xi)\right)_{j=1}^ \infty $ an
orthonormal family of real-valued eigenfunctions of the operator $\mathcal
L[\xi]$ from \eqref{L-dop}, i.e.
$$\left\{
\begin{array}{l}
-u''_j(t;\xi)+(t-\xi)^2u_j(t;\xi)  =\mu_j(\xi)u_j(t;\xi)\,,
\quad {\rm in~ }\mathbb R_+\,,\\
u_j'(0;\xi)=0\,,\\
\displaystyle\int_{\mathbb R_+}u_j(t;\xi)^2 \, d  t=1\,.
\end{array}
\right.
$$
Let us define a bounded function $\mathbb R\times\mathbb R_+\ni
(s,t)\mapsto v_{j}(s,t;\xi)$ by~:
\begin{equation}\label{phi-j}
v_{j}(s,t;\xi)=\exp\left(-i\xi\, s\right) u_j\left(t;\xi\right)\,.
\end{equation}
We introduce a family of projectors
$\Pi_j(\xi)$
on the functions $v_{j}$\,,
\begin{equation}\label{coherent-states0}
C_0^\infty(\overline{\mathbb R\times\mathbb R_+})
\ni\varphi\mapsto\int_{\mathbb R\times\mathbb R_+}
v_{j}(x_1,x_2;\xi)\,
\overline{v_{j}(y_1,y_2;\xi)}\,\,\varphi(y_1,y_2) \, d  y_1 d
y_2\,.
\end{equation}
The projectors $\Pi_j$ are kernel operators. For $j\in {\mathbb N}$, the
integral kernels are defined by:
\begin{equation}\label{kernels0}
K_j(\xi)\big{(}(x_1,x_2),(y_1,y_2)\big{)}
=v_{j}(x_1,x_2;\xi)\times\overline{v_{j}(y_1,y_2;\xi)}\,.
\end{equation}
Now, one easily verifies the following properties~:
\begin{equation}\label{Prop10}
\sum_{j=1}^\infty \int_{\R} \Pi_j( \xi)\, d  \xi=2\pi\, {\rm
Id}_{L^2(\mathbb
  R\times\mathbb R_+)}\,,
\end{equation}
and
\begin{equation}\label{Prop20New} P_{1,1,\mathbb R\times\mathbb
R_+}\,\Pi_j(\xi)= \mu_j(\xi)\Pi_j(\xi)\,.
\end{equation}
By means of a dilation, we get a family of eigenprojectors for the operator $P_{h,b,\mathbb R\times\mathbb R_+}$. Let us
introduce the unitary operator,
\begin{equation}\label{dialtion}
U_{h,b}~:L^2(\mathbb R\times\mathbb R_+)\ni\varphi \mapsto
U_{h,b}\varphi\in L^2(\mathbb R\times\mathbb R_+)\,,
\end{equation}
such that, for all $x=(x_1,x_2)\in\mathbb R\times\mathbb R_+$,
$$\left(U_{h,b}\varphi\right)(x)=\sqrt{b/h} \,\varphi(\sqrt{b/h} \, x)\,.$$
Notice that
\begin{align}
U_{h,b}^{-1} P_{h,b,\mathbb R\times\mathbb R_+} U_{h,b} = hb P_{1,1,\mathbb R\times\mathbb R_+}\,.
\end{align}
Then we introduce the family of projectors,
\begin{equation}\label{coherent-states}
\Pi_j(h,b;\xi)=U_{h,b}\Pi_j(\xi)U_{h,b}^{-1}\,.
\end{equation}
Again, the projectors $\Pi_j(h,b;\xi)$ are kernel operators. For
$j\in {\mathbb N}$, the integral kernels are given via the kernels
(\ref{kernels0}):
\begin{multline}\label{kernels}
K_j(h,b;\xi)\big{(}x,y\big{)} =\frac{b}{h}
K_j(\xi)\left(\sqrt{b/h}\,x,\sqrt{b/h}\,y\right)\\
=\frac{b}{h} e^{-i \sqrt{b/h}\,\xi(x_1-y_1)} u_j(\sqrt{b/h}\,x_2;\xi) \overline{u_j(\sqrt{b/h}\,y_2;\xi)}
\,,
\end{multline}
for all $x,y\in\mathbb
R\times\mathbb R_+$.
Now, the following properties are directly inferred from (\ref{Prop10})
and (\ref{Prop20New})~:
\begin{equation}\label{Prop1}
\sum_{j=1}^\infty \int_{\mathbb R}\Pi_j(h,b;\xi)\, d  \xi =2\pi \,
{\rm Id}_{L^2(\mathbb
  R\times\mathbb R_+)}\,,
\end{equation}
and
\begin{equation}\label{Prop2}
P_{h,b,\mathbb R\times\mathbb R_+}\,\Pi_j(h,b;\xi) =hb\,
\mu_j(\xi)\Pi_j(h,b;\xi)\,.
\end{equation}

\section{Proof of Theorem~\ref{FoKa-thm1}}\label{sec-proof}

\subsection{Existence of discrete spectrum}
Recall the operator $P_{h,\Omega}$ introduced in (\ref{JMP07Op}).
The following
inequality holds (see \cite{AHS})
for compactly supported functions
$$\int_{\Omega}|(h\nabla-i\Ab)\varphi|^2\, d  x\geq h\int_\Omega
B(x)|\varphi|^2\, d  x\,,\quad\forall~\varphi\in
C_0^\infty(\Omega)\,.$$ Using then a `magnetic' version of Persson's
Lemma (see \cite{Bon,Pe}), we get that
\begin{align*}
\inf \Spec_{\rm ess} P_{h,\Omega} \geq bh \,,
\end{align*}
hence proving
the first
statement of Theorem~\ref{FoKa-thm1}.
Let $\{e_j(h)\}$ be the sequence of eigenvalues (counted with
multiplicity) corresponding to
$\sigma(P_{h,\Omega})\cap]-\infty,bh[$\,, and define
$H=P_{h,\Omega}-bh$. It suffices now to show that
$$\sum_{j}\langle f_j\,,Hf_j\rangle>-\infty\,,$$
where the functions $f_j$ are $L^2$ normalized eigenfunctions
associated with the eigenvalues $e_j$.
To that end, we introduce a partition of unity of $\mathbb R$,
\begin{equation}\label{partition-R}
\psi_1^2+\psi_2^2=1,\quad{\rm supp}\,\psi_1\subset]-\infty,1[,\quad
{\rm supp}\,\psi_2\subset[\frac12,\infty[\,,\end{equation} and set
for $k=1,2$, $\chi_k(x)=\psi_k(t(x))$, $x\in\mathbb R^2$, where
$t(x)$ is the signed distance to $\partial\Omega$,
$$t(x)={\rm
dist}(x,\partial\Omega)\quad{\rm if~}x\in\Omega\,,\quad t(x)=-{\rm
dist}(x,\partial\Omega)\quad{\rm otherwise}.$$  By the IMS formula,
we write,
\begin{equation}\label{IMS-h}
\langle f_j\,,\,Hf_j\rangle=\sum_{k=1}^2\left( \langle  \chi_k
f_j\,,\,H\chi_kf_j\rangle -h^ 2\|\, |\nabla\chi_k|u\|^ 2\right),
\end{equation}
where $\|\cdot\|$ denotes the $L^2$ norm in $\Omega$.
Using the bound $B\geq b$ together with the fact that
$\psi_1^2+\psi_2^2=1$, we get a further decomposition of
(\ref{IMS-h}),
\begin{align}\label{IMS1-h}
\langle f_j\,,\,Hf_j\rangle=\sum_{k=1}^2 \langle
f_j\,,\,\chi_{k}(H-V)\chi_{k}f_j\rangle
\geq \sum_{k=1}^2 \langle
f_j\,,\,\chi_{k}(H_0-V)\chi_{k}f_j\rangle\,,
\end{align}
where $H_0=-(h\nabla-i\Ab)^2-hB$ and
\begin{equation}\label{potential-1}
V=h^2\left(|\nabla\chi_{1}|^2+|\nabla\chi_{2}|^2\right)\,.
\end{equation}
Pick an arbitrary
positive integer $N\leq {\rm Card}(\{e_j\}_j)$. Let us define the trial density matrix  $L^2(\mathbb R^2)\ni
f\mapsto\gamma_2 f\in L^2(\mathbb R^2)$,
$$\gamma_2 f=\chi_2\sum_{j=1}^N\langle
\chi_2f\,,\,f_j\rangle\,f_j\,,$$ which verifies the conditions
$0\leq\gamma_2\leq 1$ (in the sense of quadratic forms) and
$(H_B-V)\gamma_2$ is trace class (actually this is a finite-rank
operator). Moreover, using Lemma~\ref{thm-ES}, we know that the
operator $(H_{\mathbb R^2}-V)\mathbf 1_{]-\infty,0[}(H_{\mathbb
  R^2}-V)$ is trace-class, where $H_{\mathbb R^2}=-(h\nabla-iA)^2-hB$
is now the corresponding  Pauli operator acting in
$L^2(\mathbb R^2)$. Therefore, we deduce from Lemma~\ref{lem-VP1},
\begin{align}\label{lb-tr}\sum_{j=1}^N \langle
f_j\,,\,\chi_2(H_B-V)\chi_2f_j\rangle &= \tr[(H_B-V) \gamma_2]
\nonumber \\
& \geq {\rm tr}\left((H_{\mathbb
R^2}-V)\mathbf 1_{]-\infty,0[}(H_{\mathbb
  R^2}-V)\right)\,,\end{align}
and we notice that  this bound is uniform in $N$.

To take care of the boundary contribution, we set $\omega={\rm
int}(\Omega\cap{\rm supp}\chi_1)$ and define the trial density matrix
$$\gamma_1~: L^2(\omega)\ni
f\mapsto\gamma_1 f=\chi_1\sum_{j=1}^N\langle
\chi_1f\,,\,f_j\rangle\,f_j\in L^2(\omega)\,.$$
Thereby, we also get,
\begin{align*}
\langle f_j\,,\,\chi_1(H_B-V)\chi_1f_j\rangle &
= \tr[ (H_B-V) \gamma_1] \\
&\geq {\rm tr}\left((H_{\omega}-V)\mathbf 1_{]-\infty,0[}(H_{\omega}-V)\right)\,,
\end{align*}
where $H_\omega$ is the restriction of $H_B$ on $\omega$ with
Dirichlet condition on\break $\Omega\cap\partial({\rm supp}\chi_1)$.
One should notice that, since $\omega$ is bounded, $H_\omega-V$ has
compact resolvent and hence $(H_{\omega}-V)\mathbf
1_{]-\infty,0[}(H_{\omega}-V)$ is evidently
trace-class.
Coming back to (\ref{IMS-h}) and (\ref{IMS1-h}), we deduce that,
\begin{multline*}
-\sum_j[e_j(h)-bh]_-=\sum_j\langle f_j\,,\,Hf_j\rangle \geq \,{\rm
tr}\left((H_{\mathbb R^2}-V)\mathbf 1_{]-\infty,0[}(H_{\mathbb
  R^2}-V)\right)\\+{\rm tr}\left((H_{\omega}-V)\mathbf
  1_{]-\infty,0[}(H_{\omega}-V)\right)\,,\end{multline*}
  proving thus that the sequence $e_j(h)-bh$ is summable.

\subsection{Lower bound}
Let $\{f_1,f_2,\ldots,f_N\}$ now be any $L^2$ orthonormal set
in $D(H)$. We will give a uniform lower bound to
$$\sum_{j=1}^N\langle f_j\,,\,Hf_j\rangle\,.$$
Using
Lemma~\ref{lem-VP2}, this will imply a lower bound to ${\rm tr}
\left( H \mathbf 1_{]-\infty,0[}(H)\right)$.%

\paragraph{\it Step~1. Localization to the boundary.}\ \\
Let $\tau(h)\in]0,1[$ be a small number to be chosen later. Using
the partition of unity in (\ref{partition-R}), we put
$$\psi_{1,h}(x)=\psi_1\left(\frac{t(x)}{\tau(h)}\right),\quad
\psi_{2,h}(x)=\psi_2\left(\frac{t(x)}{\tau(h)}\right),
\quad\forall~x\in\overline\Omega.$$ By the IMS formula, we write,
\begin{equation}\label{IMS1}
\langle f_j\,,\,Hf_j\rangle=\sum_{k=1}^2 \langle
f_j\,,\,\psi_{k,h}(H-V_{h})\psi_{k,h}f_j\rangle\,,
\end{equation}
where
\begin{equation}\label{potential-2}
V_{h}=h^2\left(|\nabla\psi_{1,h}|^2+|\nabla\psi_{2,h}|^2\right)\,.
\end{equation}
Notice that the term corresponding to $k=2$ in (\ref{IMS1})
corresponds to the interior term. We will prove that it is a lower order error term.
Similarly to \eqref{lb-tr}, one can show that,
$$\langle f_j\,,\,\psi_{1,h}(H-V_h)\psi_{1,h}f_j\rangle\geq {\rm
tr}\left((H_{\mathbb R^2}-V_h)\mathbf 1_{]-\infty,0[}(H_{\mathbb
  R^2}-V_h)\right)\,,$$
and we remind the reader that $H_{\mathbb
R^2}=-(h\nabla-iA)^2-hB(x)$
is now the corresponding  self-adjoint operator acting in $L^2(\mathbb R^2)$.
Using the Lieb-Thirring inequality of Lemma~\ref{thm-ES}, we get
\begin{eqnarray*}
&&\hskip-0.5cm \sum_{j=1}^N
\langle  f_j\,,\,\psi_{2,h}(H-V_{h})\psi_{2,h}f_j\rangle\\
&&\hskip0.5cm\geq -C\left(\int_{\mathbb R^2}\left(
\|B\|_{L^\infty}[-h^{-1}V_h]_-+[-h^{-1}V_h]_-^2\right)\, d
x\right)\,.
\end{eqnarray*} 
With our potential $V_h$ from
(\ref{potential-2}), the integral on the right side above becomes
of the order of $\frac{h}{\tau(h)}+\frac{h^2}{\tau(h)^3}$. Thus, we
get
\begin{multline}\label{int-estimate}
\sum_{j=1}^N\langle f_j\,,\,Hf_j\rangle\geq \sum_{j=1}^N\langle
f_j\,,\,\psi_{1,h}
(H-V_h)\psi_{1,h}f_j\rangle\\
-C\frac{h}{\tau(h)}\left(1+\frac{h}{\tau(h)^2}\right)\,.
\end{multline}
Later, we shall choose $\tau(h)$  in such a manner that the first
term on
the right hand side above is the dominant term.

\noindent\paragraph{\it Step~2. Boundary analysis.}\ \\
Here we consider the boundary term,
$$\sum_{j=1}^N\langle
f_j\,,\,\psi_{1,h} (H-V_h)\psi_{1,h}f_j\rangle\,.$$ To that  end,
let $\chi \in L^2(\R)$ be positive, smooth, supported in $]0,1[$ and
satisfying \begin{equation}\label{chi-normalized} \int \chi^2(s)
\,ds = 1\,.\end{equation} Using the boundary coordinates $(s,t)$
introduced in (\ref{phi0}), we put
$$\chi_h(x;\sigma)=\frac1{\sqrt{\tau(h)}}\,\chi\left(\frac{s(x)-\sigma}{\tau(h)}\right)\,,\quad
\forall~x\in\Omega(t_0),~\forall~\sigma\in\R\,,
$$
and we notice that, for all $\sigma\in\R\setminus
\,]-\tau(h),|\partial\Omega|\,[$,
$$\chi_h(x;\sigma)=0\quad\forall~x\in\Omega(t_0)\,.$$ Using again an IMS type
decomposition, we write,
\begin{align}\label{eq:IntSigma}
&\sum_{j=1}^N\langle f_j\,,\,\psi_{1,h}
(H-V_h)\psi_{1,h}f_j\rangle\nonumber\\
&\qquad= \int_{\R}\sum_{j=1}^N\left\langle
f_j\,,\,\psi_{1,h}\chi_{h}(x;\sigma)
(H-W_h)\psi_{1,h}\chi_h(x;\sigma)\right\rangle
\, d \sigma\,,
\end{align}
where
\begin{equation}\label{potential2}
W_h=V_h+h^2\int_{\R}|\nabla\chi_h(x;\sigma)|^2\, d  \sigma\,.
\end{equation}
Let us denote by ($\Phi_{t_0}$ is the coordinate change (\ref{phi0})
valid near the boundary)
\begin{equation}\label{ujh}
u_{j,h}(\sigma)=u_{j,h}(x;\sigma)=\psi_{1,h}(x)\chi_{h}(x;\sigma)f_j(x)\,,\quad
B_\sigma=B(\Phi_{t_0}^{-1}(\sigma,0))\,.
\end{equation}
Using Lemma~\ref{VI-FrLem3.5}, we get a constant $C>0$ and a scalar
function $\phi=\phi_\sigma$ such that, for all
$C\tau(h)\leq\varepsilon\ll1$, we have,
\begin{multline}\label{eq:epsilon}
\sum_{j=1}^N\left\langle
f_j\,,\,\psi_{1,h}\chi_h(\sigma)
(H-W_h)\psi_{1,h}\chi_h(\sigma)\right\rangle
\geq (1-\varepsilon) \\
\times \sum_{j=1}^N\langle
e^{i\phi_\sigma/h}\widetilde u_{j,h}(\sigma)\,,\,
(H_{h,B_\sigma,\mathbb R\times\mathbb R_+}-\mathcal
W_h)e^{i\phi_\sigma/h} \widetilde u_{j,h}(\sigma)\rangle
_{L^2(\mathbb R\times\mathbb R_+)}\,,
\end{multline}
where
$$H_{h,B_\sigma,\mathbb R\times\mathbb R_+}=P_{h,B_\sigma,\mathbb R\times\mathbb
  R_+}
-bh=-(h\nabla-iB_\sigma\Ab_0)^2-bh\,,$$ $\Ab_0$ is the potential
introduced in (\ref{A0}),
\begin{equation}\label{potential3}
\mathcal W_h=\widetilde W_h+C\varepsilon^{-1}
\left(\tau(h)^4+h^2\right)\,,
\end{equation}
and to a function $v(x)$, we associate the function $\widetilde
v(s,t)$ by means of (\ref{VI-utilde}). The functions $\widetilde
u_{j,h}(\sigma)$ are naturally extended by $0$ in
$\R\times\R_+\setminus{\rm supp}\,\widetilde u_{j,h}(\sigma)$.

Later, we shall make a suitable choice of the parameter
$\varepsilon$. We use part of the kinetic energy
$H_{h,B_\sigma,\R\times\R_+}$ to control the error resulting from
$\mathcal W_h$. Let $\delta\in]0,1[$ be another parameter to be
specified later; $\delta$ will be chosen as a function of $h$. We
decompose the previous sum  in two\,,
\begin{align}\label{s=s1+s2}
&
\sum_{j=1}^N\langle e^{i\phi_\sigma/h}\widetilde
u_{j,h}(\sigma)\,,\, (H_{h,B_\sigma,\mathbb R\times\mathbb
R_+}-\mathcal W_h)e^{i\phi_\sigma/h} \widetilde
u_{j,h}(\sigma)\rangle
_{L^2(\mathbb R\times\mathbb R_+)}\\
&= (1-\delta)\sum_{j=1}^N \langle
e^{i\phi_\sigma/h}\widetilde u_{j,h}(\sigma)\,,\,
H_{h,B_\sigma,\mathbb R\times\mathbb R_+}e^{i\phi_\sigma/h}
\widetilde u_{j,h}(\sigma)\rangle
_{L^2(\mathbb R\times\mathbb R_+)}\nonumber\\
&\quad+\delta\sum_{j=1}^N \langle
e^{i\phi_\sigma/h}\widetilde u_{j,h}(\sigma)\,,\,
(H_{h,B_\sigma,\mathbb R\times\mathbb R_+} -\delta^{-1}\mathcal
W_h)e^{i\phi_\sigma/h} \widetilde u_{j,h}(\sigma)\rangle
_{L^2(\mathbb R\times\mathbb R_+)}\,.\nonumber
\end{align}
Let us estimate the last term on the right side.
  We take $T=T(h)\gg1$, to be specified
later, and we make the following choice of $\tau(h)$,
$$\tau(h)=h^{1/2}T\,.$$
Notice now that \begin{equation}\label{bound-W}
\|\delta^{-1}\mathcal W_h\|_{L^\infty(\mathbb R\times\mathbb R_+)}
\leq
C\delta^{-1}\left(\varepsilon^{-1}h^2T^4+hT^{-2}\right)\leq\lambda
B_\sigma h\,,\end{equation} where we have defined
\begin{equation}\label{eq-lambda}
\lambda=Cb^{-1}\delta^{-1}\left(\varepsilon^{-1}hT^4
+T^{-2}\right)\,.\end{equation} 
Furthermore, with $S=\tau(h)$ and
the above notation,
 we define the operator $\widetilde \gamma$ on $L^2([0,S]\times]0,h^{1/2}T[)$,
$$\widetilde \gamma f = \sum_{j=1}^N
\langle f\,,\,e^{i\phi_\sigma /h}\widetilde
u_{j,h}(\sigma)\rangle_{L^2([0,S]\times]0,h^{1/2}T[)}
e^{i\phi_\sigma/h}\widetilde u_{j,h}(\sigma)\,.$$ The next estimate
is a simple application of
Proposition~\ref{App:transf},
\begin{align}\label{eq:GenDensMat}
0&\leq \langle f,\widetilde \gamma f\rangle  \\
&\leq(1+C\tau(h))\int_{\Omega}\chi_h(x;\sigma)^2\,\psi_{1,h}^2(x) \left|\left(f
    \circ \Phi_{t_0}^{-1}\right)(x)\right|^2\,d
x\leq C'\tau(h)^{-1}\|f\|_{L^2}^2\,.\nonumber
\end{align} 
Thus, $\frac{1}{C'\tau(h)^{-1}} \widetilde \gamma$ is a density matrix.
The operator $\widetilde\gamma$ is constructed so that we may write the
term we wish to estimate in the following form,
\begin{multline*}
\sum_{j=1}^N \langle e^{i\phi_\sigma/h}\widetilde
u_{j,h}(\sigma)\,,\, (H_{h,B_\sigma,\mathbb R\times\mathbb R_+}
-\delta^{-1}\mathcal W_h)e^{i\phi_\sigma/h} \widetilde
u_{j,h}(\sigma)\rangle _{L^2(\mathbb R\times\mathbb R_+)}\\
={\rm tr}\,\left[(H_{h,B_\sigma,\mathbb R\times\mathbb R_+}
-\delta^{-1}\mathcal W_h)\widetilde \gamma\right]\,.\end{multline*}
Using the bound  (\ref{bound-W}) on the potential $\mathcal W_h$ and
the variational principle of Lemma~\ref{lem-VP1}, we get,
\begin{eqnarray}\label{vp-variant} &&\hskip-1cm\sum_{j=1}^N \langle
e^{i\phi_\sigma/h}\widetilde u_{j,h}(\sigma)\,,\,
(H_{h,B_\sigma,\mathbb R\times\mathbb R_+} -\delta^{-1}\mathcal
W_h)e^{i\phi_\sigma/h} \widetilde u_{j,h}(\sigma)\rangle
_{L^2(\mathbb R\times\mathbb
R_+)}\nonumber\\
&&\geq -C'\tau(h)^{-1}\,\mathcal E(\lambda,B_\sigma,S,T)\,.
\end{eqnarray}
Here we remind the reader that $S=\tau(h)=h^{1/2}T$, the energy
$\mathcal E(\lambda,B_\sigma,S,T)$ is introduced in
\eqref{energy-torus}, and $C'$ is a positive constant depending only
on $\Omega$.

In order to use the estimate given in Lemma~\ref{roughestimate}, we
make the following choice of $\varepsilon$ and $\delta$,
$$
T=h^{-1/8}\,, \qquad \varepsilon=h T^6 = h^{1/4}\,,\qquad \delta=T^{-3/2}=h^{3/16}\,,$$ so that the following conditions hold~:
$$h^{1/2}T\ll\varepsilon\ll1\,,\qquad \lambda\ll1\,.$$
Notice that this choice implies
\begin{align}
  \label{eq:4}
  \tau(h) = h^{3/8}.
\end{align}

Using now Lemma~\ref{roughestimate}, we conclude that (recall that
$S=\tau(h)$ in our case which compensates the factor $\tau^{-1}$ from \eqref{eq:GenDensMat}),
\begin{align*}
&\sum_{j=1}^N\langle e^{i\phi_\sigma/h}\widetilde
u_{j,h}(\sigma)\,,\, (H_{h,B_\sigma,\mathbb R\times\mathbb R_+}
-\delta^{-1}\mathcal W_h)e^{i\phi_\sigma/h} \widetilde
u_{j,h}(\sigma)\rangle
_{L^2(\mathbb R\times\mathbb R_+)}\\
&\qquad\geq -C h^{1/2}TB_\sigma^{1/2}\,,
\end{align*}
and
consequently, after integration w.r.t. $\sigma\in]-\tau(h),|\partial\Omega|\,[\,$,
\begin{align}\label{bnd-error}
&\delta\int\sum_{j=1}^N\langle
e^{i\phi_\sigma/h}\widetilde u_{j,h}(\sigma)\,,\,
(H_{h,B_\sigma,\mathbb R\times\mathbb R_+} -\delta^{-1}\mathcal
W_h)e^{i\phi_\sigma/h} \widetilde u_{j,h}(\sigma\rangle
_{L^2(\mathbb R\times\mathbb R_+)} d \sigma\nonumber\\
&\qquad\geq -C\delta T h^{1/2}=-C h^{9/16}\,,
\end{align}
thus obtaining an error of the order $o(h^{1/2})$. We point also out
that with our choice, $\tau(h)=h^{3/8}$, so that the error in
(\ref{int-estimate}) becomes small,  of the order of
$h^{5/8}=o(h^{1/2})$.

\paragraph{\it Step~3. Estimating the leading order term.}\ \\
Now we estimate the first term on the right hand side of
(\ref{s=s1+s2}). We use the projectors $\Pi_j$ constructed in
Section~\ref{sec-coherentstates}. In view of (\ref{Prop1}) and
(\ref{Prop2}), we have the following splitting,
\begin{align*}
H_{h,B_\sigma,\mathbb R\times\mathbb
R_+}&=\frac1{2\pi}\sum_{p=1}^\infty\int_{\R}H_{h,B_\sigma,\mathbb
R\times\mathbb R_+}\Pi_p(h,B_\sigma;\xi)\, d \xi\\
&=\frac{B_\sigma h}{2\pi}\sum_{p=1}^\infty\int_{\R}
\left(\mu_p(\xi)-\frac{b}{B_\sigma}\right)\Pi_p(h,B_\sigma;\xi)\, d \xi\,.
\end{align*}
Since $B_\sigma\geq b$, Lemma~\ref{RuFr-Lemm-kach} gives that for
$p\geq 2$, $\mu_p(\xi)-\frac{b}{B_\sigma}>0$. Hence, we get for any function $f$ in the domain of
$H_{h,B_\sigma,\mathbb R\times\mathbb R_+}$,
$$\langle  f\,,\,H_{h,B_\sigma,\mathbb R\times\mathbb
R_+}f\rangle\geq -\frac{B_\sigma h}{2\pi}\int_{\R}
\left[\mu_1(\xi)-\frac{b}{B_\sigma}\right]_-\langle
f\,,\,\Pi_1(h,B_\sigma;\xi)f\rangle\, d \xi\,.
$$
Therefore, defining \begin{equation}\label{SN-sx} \mathcal
S_N(\sigma;\xi)=\sum_{j=1}^N\left\langle
e^{i\phi_\sigma/h}\widetilde u_{j,h}(\sigma)\,,\,
\Pi_1(h,B_\sigma;\xi)e^{i\phi_\sigma/h} \widetilde
u_{j,h}(\sigma)\right\rangle _{L^2(\mathbb R\times\mathbb R_+)}\geq
0,\end{equation} we get that
\begin{multline}\label{SN-sx1}
 \sum_{j=1}^N\left\langle e^{i\phi_\sigma/h}\widetilde
u_{j,h}(\sigma)\,,\, H_{h,B_\sigma,\mathbb R\times\mathbb
R_+}e^{i\phi_\sigma/h} \widetilde u_{j,h}(\sigma)\right\rangle
_{L^2(\mathbb R\times\mathbb R_+)}\\
\geq - \frac{B_\sigma h}{2\pi}\int_{\mathbb
R}\left[\mu_1(\xi)-\frac{b}{B_\sigma}\right]_-\,\mathcal
S_N(\sigma;\xi)\, d \xi\,.
\end{multline}
We estimate now the term on the right hand side above. We start
first by estimating $S_N(\sigma;\xi)$. Recalling the definition of
$u_{j,h}$, and using again the coordinate transformation valid near
the boundary, we get (using the choice $\tau(h) = h^{3/8}$),
\begin{multline}\label{SN-sx2}
\left\langle e^{i\phi_\sigma/h}\widetilde u_{j,h}(\sigma)\,,\,
\Pi_1(h,b_\sigma;\xi)e^{i\phi_\sigma/h} \widetilde
u_{j,h}(\sigma)\right\rangle
_{L^2(\mathbb R\times\mathbb R_+)}\\
\leq (1+Ch^{3/8})\left\langle f_j\,,\,\mathcal
P(h,B_\sigma;\sigma,\xi)\,f_j \right\rangle_{L^2(\Omega)}\,.
\end{multline}
Here $\mathcal P(h,B_\sigma;\sigma,\xi)$ is a positive operator,
which is given by,
$$\mathcal P(h,B_\sigma;\sigma,\xi)=
\psi_{1,h}\chi_{h}(\sigma)\,U_{\Phi}^{-1}
e^{-i\phi_\sigma/h}\,\Pi_1(h,B_\sigma;\xi)\,e^{i\phi_\sigma/h}\,
U_{\Phi}\psi_{1,h}\chi_{h}(\sigma)\,,$$ and the transformation $U_{\Phi}$ is associated to the coordinate change
$\Phi_{t_0}$ introduced in (\ref{phi0}).
Since $\{f_1,f_2,\ldots,f_N\}$ is an
orthonormal family in $L^2(\Omega)$, we deduce that
\begin{eqnarray*}
&&\hskip-1cm\sum_{j=1}^N\left\langle f_j\,,\,\mathcal
P(h,B_\sigma;\sigma,\xi)\,f_j\right
\rangle_{L^2(\Omega)}\\
&&\leq {\rm tr}\,\big{(}\mathcal P(h,B_\sigma;\sigma,\xi)\big{)}\\
&&=\int_{\mathbb R\times\mathbb R_+}
\frac{B_\sigma}{h}|\chi_{h}(s;\sigma)|^2|\psi_{1,h}(t)|^2|u_1(h^{-1/2}B_\sigma^{1/2}t;\xi)|^2\,(1-tk(s)) d
s d  t\,,
\end{eqnarray*}
where
\begin{eqnarray*}
&&\hskip-2cm \int_{\mathbb R\times\mathbb R_+}
|\chi_{h}(s;\sigma)|^2|\psi_{1,h}(t)|^2|u_1(h^{-1/2}B_\sigma^{1/2}t;\xi)|^2\, d
s d  t\\
&&\leq\tau(h)^{-1}\int_{\mathbb R}\left|\chi
\left(\frac{s-\sigma}{\tau(h)}\right)\right|^2\, d  s\int_{\mathbb
R_+}
|u_1(h^{-1/2}B_\sigma^{1/2}t;\xi)|^2\, d  t\\
&&=h^{1/2}B_\sigma^{-1/2}\,,
\end{eqnarray*}
and similarly,
\begin{multline*}
 \int_{\mathbb R\times\mathbb R_+}
|\chi_{h}(s;\sigma)|^2|\psi_{1,h}(t)|^2|u_1(h^{-1/2}B_{\sigma}^{1/2}t;\xi)|^2\,
tk(s) d  s d  t\\
= \mathcal O(h^{1/2}\tau(h))=\mathcal O(h^{7/8})\,.
\end{multline*}
Coming back to (\ref{SN-sx}) and (\ref{SN-sx2}), we get that,
$$0\leq S_N(\sigma;\xi)\leq \frac{B_\sigma}{h^{1/2}}+\mathcal
O(h^{-1/8})\,.$$ Implementing the above estimate in (\ref{SN-sx1}),
we deduce the following lower bound,
\begin{multline*}
 \sum_{j=1}^N\left\langle e^{i\phi_\sigma/h}\widetilde
u_{j,h}(\sigma)\,,\, H_{h,B_\sigma,\mathbb R\times\mathbb
R_+}e^{i\phi_\sigma/h} \widetilde u_{j,h}(\sigma)\right\rangle
_{L^2(\mathbb R\times\mathbb R_+)}
\\
\geq -\frac{h^{1/2}B_\sigma^{3/2}}{2\pi} \int_{\mathbb
R}\left[\mu_1(\xi)-\frac{b}{B_\sigma}\right]_- d \xi-\mathcal
O(h^{7/8})\,,
\end{multline*}
and consequently, upon integrating with respect to $\sigma$ (recall
that $\widetilde u_{j,h}(\sigma)$ is supported in $\{0\leq
s-\sigma\leq \tau(h)~:s\in [0,|\partial\Omega|\,[\,\}$),
\begin{align}\label{eq:MainTerm}
&\int \sum_{j=1}^N\left\langle e^{i\phi_\sigma/h}\widetilde
u_{j,h}(\sigma)\,,\, H_{h,B_\sigma,\mathbb R\times\mathbb
R_+}e^{i\phi_\sigma/h} \widetilde u_{j,h}(\sigma)\right\rangle
_{L^2(\mathbb R\times\mathbb R_+)}\,d\sigma\nonumber \\
& \qquad\geq-
\int_{0}^{|\partial\Omega|}\frac{h^{1/2}B_\sigma^{3/2}}{2\pi}\int_{\mathbb
R}\left[\mu_1(\xi)-\frac{b}{B_\sigma}\right]_-
 d \xi d \sigma-\mathcal O(h^{5/8})\,.
\end{align}
Recalling the definition of $B_\sigma=B(\Phi_{t_0}^{-1}(\sigma,0))$,
the integral on the r.h.s. above is actually nothing but
$$\frac{h^{1/2}}{2\pi}\int_{\partial\Omega} \left( B(x)^{3/2}\int_{\mathbb
R}\left[\mu_1(\xi)-\frac{b}{B(x)}\right]_- d \xi \right) d  x\,.$$
Therefore, when collecting \eqref{int-estimate}, \eqref{eq:IntSigma}, \eqref{eq:epsilon}, \eqref{s=s1+s2}, \eqref{bnd-error} and \eqref{eq:MainTerm}, we deduce finally the desired bound,
\begin{align}\label{conclusion}
&\sum_{j=1}^N\langle f_j\,,\,Hf_j\rangle\nonumber\\
&\qquad\geq -\frac{h^{1/2}}{2\pi}\int_{\partial\Omega} \int_{\mathbb R}
B(x)^{3/2}\left[\mu_1(\xi)-\frac{b}{B(x)}\right]_- \, d  \xi d
x-\mathcal O(h^{9/16})\,,\end{align}
uniformly with
respect to $N$ and the orthonormal  family $\{f_j\}$.

\subsection{Upper bound}\label{sec-thm1-ub}
\paragraph{{\it Coherent states for the curved boundary}}~\\
Let $\sigma \in \frac{|\partial \Omega|}{2 \pi} {\mathbb S}^1$
(identified with $\partial \Omega$ through a parametrization of the
boundary). Let $\phi = \phi_{\sigma}$ be the gauge function from
Lemma~\ref{VI-FrLem3.5}. With $\Phi = \Phi_{t_0}$ being the
coordinate change near the boundary, let $B_{\sigma} =
B(\Phi^{-1}(\sigma,0))$ i.e. the magnetic field at the boundary
point parameterized by $\sigma$. Let finally, $\chi \in
C^{\infty}(\R)$ be a positive, smooth, have $\supp \chi \subset \R_{+}$, and be normalized such that
$\int_{\R}\chi^2(s)\,ds = 1$. Define $\chi_h(s)$ for $s\in[0,|\partial\Omega|\,[$ to be
\begin{align*}
\chi_h(s):=\tau(h)^{-1/2}\chi(s/\tau(h))\,,
\end{align*}
and extend $\chi_h$ to be a
$|\partial\Omega|$-periodic function.
Setting
\begin{multline}
\widetilde{f}_j( (s,t); h, \sigma, \xi) := \\
\left(\frac{B_{\sigma}}{h}\right)^{1/4} e^{-i\xi s
\sqrt{B_{\sigma}/h}}u_j\left(  \sqrt{\frac{B_{\sigma}}{h}}\,t;
\xi\right) e^{-i \phi_{\sigma}/h} \chi_h(s-\sigma)\psi_{1,h}(t)\,,
\end{multline}
we get---by the coordinate transformation $\Phi$---the following
function in $\Omega(t_0)$,
\begin{align}
f_j(x; h, \sigma, \xi) = \widetilde{f}_j( \Phi(x); h, \sigma, \xi).
\end{align}
We will consider $f_j$ as a function on all of $\Omega$ by extension
by zero to $\Omega \setminus \Omega(t_0)$. Let $\Pi_j^{\rm
bnd}(h,\sigma,\xi)$ be the operator with integral kernel
\begin{align}
\Pi_j^{\rm bnd}(h,\sigma,\xi)(x,x') = f_j(x; h, \sigma, \xi)
\overline{f_j(x'; h, \sigma, \xi)}.
\end{align}
In terms of the projectors constructed in (\ref{coherent-states}),
$\Pi_j^{\rm bnd}(h,\sigma,\xi)$ is expressed as follows,
$$
\Pi_j^{\rm
bnd}(h,\sigma,\xi)=\psi_{1,h}U_{\Phi}^{-1}e^{-i\phi_\sigma/h}\chi_h(\cdot -\sigma)\Pi_j(h,B_\sigma;\xi)e^{i\phi_\sigma/h}
\chi_h(\cdot-\sigma)U_{\Phi}\,\psi_{1,h}\,.$$
Then
\begin{align}
\tr[ &\Pi_j^{\rm bnd}(h,\sigma,\xi) ]  = \int_{\Omega}|f_j(x; h, \sigma, \xi)|^2\,dx \nonumber \\
&= \iint\big|  \widetilde{f}_j\big(  (s,t); \sigma, \xi\big) \big|^2  (1-tk(s)) \,dsdt\nonumber \\
&\leq (1 + \tau(h) \| k \|_{\infty} ) \iint \sqrt{\frac{B_{\sigma}}{h} }\,\big|  u_j\big(  \sqrt{\frac{B_{\sigma}}{h}} \,t; \xi\big) \big|^2 \tau(h)^{-1} \chi^2(\frac{s-\sigma}{\tau(h)})\,dsdt \nonumber \\
&= (1 + \tau(h) \| k \|_{\infty} )\,.
\end{align}
Also, using Lemma~\ref{VI-FrLem3.5},
\begin{align}
\tr[ &P_{h,\Omega} \Pi_j^{\rm bnd}(h,\sigma,\xi) ]\nonumber = q_{h,\Omega}\big(f_j(\cdot; h, \sigma, \xi)\big) \nonumber \\
&\leq (1+\varepsilon) \int
|(h\nabla-iB_{\sigma} \Ab_0)e^{i\phi} \widetilde{f}_j(h, \sigma, \xi)|^2 \,dsdt \nonumber \\
&\qquad
+C\varepsilon^{-1}\left(\tau(h)^4+h^2\right) \int |
\widetilde{f}_j(h, \sigma, \xi)|^2 \,dsdt \nonumber \\
&\leq (1+\varepsilon) \big(\frac{B_{\sigma}}{h}\big)^{1/2} \int \chi_h^2(s-\sigma)\nonumber\\
&\qquad\qquad\qquad\times
\Big|(h\nabla-iB_{\sigma} \Ab_0) e^{-i\xi s \sqrt{B_{\sigma}/h}}u_j\big(  \sqrt{\frac{B_{\sigma}}{h}}\,t; \xi\big)\Big|^2 \,dsdt \nonumber \\
&\qquad + \Big[ (1+\varepsilon) \frac{h^2}{\tau(h)^2} + C \varepsilon^{-1} (\tau(h)^4+h^2)\Big] \nonumber \\
&\leq (1+\varepsilon) \mu_1(\xi) h  B_{\sigma} + C \Big[  \frac{h^2}{\tau(h)^2} +\varepsilon^{-1} (\tau(h)^4+h^2)\Big]\,.
\end{align}
Let $M(h,\sigma, \xi,j)$ be a function with $0\leq M \leq 1$, and write
\begin{align}\label{den-ub}
\gamma = \sum_{j=1}^{\infty}\iint M(h,\sigma, \xi,j) \Pi_j^{\rm
bnd}(h,\sigma,\xi) \frac{d\sigma d\xi}{2 \pi \sqrt{h/B_{\sigma}}}\,.
\end{align}
Then, clearly $0 \leq \gamma$ as an operator on $L^2(\Omega)$. We will prove that
\begin{align}
\label{eq:BdI}
\gamma \leq (1 + \| k\|_{\infty} \tau(h))\,.
\end{align}
Consider $g \in L^2(\Omega)$ and define $g_{\partial} = g 1_{\{\dist(x,\partial \Omega) \leq \tau(h)\}}$.
Then
\begin{align*}
\langle g\,|\,\gamma g \rangle = \langle g_{\partial}\,|\,\gamma g_{\partial} \rangle,
\end{align*}
so we may assume that $g = g_{\partial}$, i.e. that $\supp g \subset \{ x\,:\,\dist(x,\partial \Omega) \leq \tau(h)\}$.
We calculate,
\begin{multline}\label{eq:SmallerId}
\langle g\,|\,\gamma g \rangle = \sum_j
\iint M(h,\sigma, \xi,j) \\
\times \left| \int \overline{\widetilde g(s,t)} \widetilde{f}_j(s,t;h,\sigma,\xi) (1-tk(s))\,dsdt \right|^2  \frac{d\sigma d\xi}{2 \pi \sqrt{h/B_{\sigma}}}\,.
\end{multline}
We estimate from above by replacing
$ \iint M \times |\cdot|^2 $ by $\iint 1 \times |\cdot|^2$ in the above expression.
Then we use the fact that $u_j(\cdot;\xi)$ is an orthonormal basis of $L^2(\R_{+})$ for all $\xi$.
Next we evaluate the $\xi$ (Fourier) integral, which becomes
\begin{align}
\int_{-\infty}^{\infty} e^{-i\xi (s-s') \sqrt{B_{\sigma}/h}} = 2\pi \sqrt{h/B_{\sigma}}\, \delta(s-s')\,.
\end{align}
After inserting these two results, \eqref{eq:SmallerId} becomes
\begin{multline}
\langle g\,|\,\gamma g \rangle\leq
\int_{\sigma \in \partial \Omega}  \int_{s\in \partial \Omega} \int_{t\in \R_+}|\widetilde g(s,t)|^2 (1-tk(s))^2 \\
\tau(h)^{-1}\chi^2(\frac{s - \sigma}{\tau(h)}) \psi_{1,h}^2(t)\,dsdtd\sigma.
\end{multline}
We do the $\sigma$-integration first. The normalization of $\chi$ implies that the result is
\begin{align*}
\langle g\,|\,\gamma g \rangle&\leq \int_{s\in \partial \Omega} \int_{t\in \R_+}|\widetilde g(s,t)|^2 (1-tk(s))^2 \psi_{1,h}^2(t)\,dsdt\\
&\leq (1 + \tau(h) \|k \|_{\infty}) \int_{\Omega} |g(x)|^2\,dx\,.
\end{align*}
That finishes the proof of \eqref{eq:BdI}.

Let $K >0$. We choose the function $M = M^K \cdot 1_{\{j=1\}}$, with $M^K$ being the characteristic function of the set
\begin{align*}
\{ (\sigma, \xi) \in \frac{|\partial \Omega|}{2\pi} {\mathbb S}^1 \times \R \,:\, \frac{b}{B_{\sigma}} - \mu_1(\xi) \geq 0, |\xi| \leq K \}\,.
\end{align*}
Let $\gamma^K$ be the corresponding density matrix. We calculate,
\begin{align*}
\tr[ (P_{h,\Omega} &- bh) \gamma^K ] \\&=
\iint M^K(h,\sigma, \xi) \big( q_{h,\Omega}( f_1(\cdot;h,\sigma,\xi)) - bh \| f_1(\cdot; h,\sigma,\xi)\|_{L^2(\Omega)} \big)\\
&\leq
\iint M^K(h,\sigma, \xi)
\Big\{ (1+\varepsilon) \mu_1(\xi) h B_{\sigma} - bh (1 - \tau(h) \| k\|_{\infty})\\
&\qquad\qquad\qquad\qquad\qquad+ C \Big[  \frac{h^2}{\tau(h)^2} +\varepsilon^{-1} (\tau(h)^4+h^2)\Big] \Big\}\frac{d\sigma d\xi}{2 \pi \sqrt{h/B_{\sigma}}}\\
&\leq
-\int_{-K}^K \int_0^{|\partial \Omega|} \big[ \mu_1(\xi) h B_{\sigma} - bh\big]_{-} \frac{d\sigma d\xi}{2 \pi \sqrt{h/B_{\sigma}}} \\
&\qquad+
h ( \varepsilon \| B \|_{L^{\infty}(\partial \Omega)} + \tau(h) \| k \|_{\infty} ) |\partial \Omega| 2K\frac{\sqrt{\| B \|_{L^{\infty}(\partial \Omega)} }}{2\pi \sqrt{h}}
\\
&\qquad+
C \Big[  \frac{h^2}{\tau(h)^2} +\varepsilon^{-1} (\tau(h)^4+h^2)\Big] |\partial \Omega| 2K \frac{\sqrt{\| B \|_{L^{\infty}(\partial \Omega)} }}{2\pi \sqrt{h}}\,.
\end{align*}
The choice
\begin{align*}
\tau(h) = h^{3/8}\,,\qquad \varepsilon = h^{1/4}\,,
\end{align*}
yields an error term which is controlled by $C K h^{3/4}$. Therefore,
\begin{multline}
\tr[ (P_{h,\Omega} - bh) \gamma^K ] \\
\leq - \frac{\sqrt{h}}{2 \pi} \int_{-K}^K \int_0^{|\partial \Omega|} B_{\sigma}^{3/2} \big[ \frac{b}{B_{\sigma}} - \mu_1(\xi) \big]_{+} d\sigma d\xi  + C K h^{3/4}\,.
\end{multline}
Since $K$ can be chosen arbitrarily large, this implies the upper bound in Theorem~\ref{FoKa-thm1}.

\section{Proof of Theorem~\ref{thm2-FK}}\label{sec-proof-thm2}
\subsection{Lower bound}
We follow the strategy of the proof of Theorem~\ref{FoKa-thm1}. Let
$H=P_{h,\Omega}-bh-a h^{3/2}$. We seek a uniform lower bound
of
$$\langle  f_j\,,\, Hf_j\rangle\,,$$
for any orthonormal family $\{f_1,f_2,\ldots, f_N\}$ in $D(H)$.
Using the notation of Section~\ref{sec-proof}, we write (see
(\ref{IMS1})),
\begin{align}\label{*}
\sum_{j=1}^N\langle f_j\,,\,Hf_j\rangle=
\sum_{k=1}^2\sum_{j=1}^N\langle f_j\,,
\psi_{k,h}(H-V_h)\psi_{k,h}f_j\rangle\,.
\end{align}
The term corresponding to $k=1$ is a boundary term and is estimated
exactly as in Section~\ref{sec-proof} (see also (\ref{conclusion})),
the extra term $ah^{3/2}$ being included in the error term. Notice
that the localization error $h^2 \tau^{-2} = h^{5/4} \gg h^{3/2}$,
so the $ah^{3/2}$-term can be incorporated without change to the
argument. Therefore, the following lower bound holds,
\begin{multline}\label{**}
\sum_{j=1}^N\langle f_j\,,
\psi_{1,h}(H-V_h)\psi_{1,h}f_j\rangle\geq\\
-\frac{h^{1/2}}{2\pi}\int_{\partial\Omega\times\R}b^{3/2}\left[\mu_1(\xi)-1\right]_-
d \xi d s(x)-\mathcal O(h^{9/16})\,.
\end{multline}

Next, we estimate the interior term corresponding to $k=2$. Let us
introduce a localization function $\chi\in C_0^\infty(B(0,1))$ such
that $\|\chi\|_{L^2(\R^2)}=1$. Define
$$\chi_h
(x;z)=\frac1{\zeta(h)}\chi\left(\frac{x-z}{\zeta(h)}\right)\,,\quad
\forall~x\in\Omega,~\forall~z \in \R^2\,,
$$
where $\zeta(h)\in]0,1[$ is to be chosen later. Let us notice that
$$\chi_h(x;z)=0\quad\forall~x\in\Omega\,,~\forall~z\in\Omega+B(0,\zeta(h))\,.$$
Then we get, using the standard localization formula,
\begin{eqnarray}\label{***}
&&\hskip-0.5cm\sum_{j=1}^N\langle f_j\,,\,\psi_{2,h}(H-V_h)\psi_{2,h}f_j\rangle\nonumber\\
& &\hskip0.5cm =\int_{\R^2} \sum_{j=1}^N\langle
f_j\,,\,\psi_{2,h}\chi_h(x;z)(H-W_h)\psi_{2,h}\chi_h(x;z)f_j\rangle\,
d z\,,\end{eqnarray} where
$W_h(x)=V_h(x)+h^2\int_{\R^2}|\nabla_x\chi(x;z)|^2 d  z$.

The term corresponding to $W_h$ will be
estimated using a Lieb-Thirring inequality. Let $\delta\in]0,1[$
to be chosen later, and write,
\begin{align}\label{****}
\langle
f_j\,,\,\psi_{2,h}\chi_h(x;z)&(H-W_h)\psi_{2,h}\chi_h(x;z)f_j\rangle\nonumber\\
&=
(1-\delta)\langle
f_j\,,\,\varphi(x;z)H\varphi(x;z)f_j\rangle
\nonumber\\
&\qquad+\delta\langle
f_j\,,\,\varphi(x;z)(H_{b,\R^2}-\mathcal \delta^{-1}\mathcal
W_h)\varphi(x;z)f_j\rangle \,,
\end{align}
where
\begin{align*}
  \mathcal W_h(x)=\mathbf 1_\Omega(x)\left(\delta a
h^{3/2}+V_h(x)+h^2\int_{\R^2}|\nabla_x\chi(x;z)|^2 d
z\right)\,,
\end{align*}
\begin{align*}
  \varphi_h(x;z)=\psi_{1,h}(x)\chi_h(x;z)\,,
\end{align*}
and $H_{b,\R^2}=-(h\nabla-ib\Ab_0)^2-bh$ is the self-adjoint
operator acting in $L^2(\R^2)$.

As explained in (\ref{lb-tr}),
 the variational
principle yields (i.e. Lemma~\ref{lem-VP1}),
\begin{multline*}
\delta\sum_{j=1}^N\langle f_j\,,\,\varphi(x;z)(H_{b,\R^2}-
\delta^{-1}\mathcal W_h)\varphi(x;z)f_j\rangle\geq\\
\delta{\rm tr}\left((H_{b,\R^2}-\delta^{-1}\mathcal W_h)\mathbf
1_{]-\infty,0[}(H_{b,\R^2}-\delta^{-1}\mathcal
W_h)\right)\,.\end{multline*} To estimate the last term, we use the
Lieb-Thirring inequality of Lemma~\ref{thm-ES}. Thus we get,
\begin{eqnarray*}
&&\hskip-0.5cm \delta\sum_{j=1}^N\langle
f_j\,,\,\varphi(x;z)(H_{b,\R^2}-\mathcal
\delta^{-1}W_h^{-1})\varphi(x;z)f_j\rangle\\
&&\hskip0.5cm\geq -C\delta\left(\int_{\mathbb R^2}\left(
b[-h^{-1}\delta^{-1}\mathcal W_h]_-+[-h^{-1}\delta^{-1}\mathcal
W_h]_-^2\right)\, d
x\right)\\
&&\hskip0.5cm \geq -C \left(\delta a h^{1/2}
+\frac{h}{\tau(h)}\left(1+\delta^{-1}\frac{h}{\tau(h)^2}\right)+\frac{h}{\zeta(h)^{2}}\left(1+\delta^{-1}
\frac{h}{\zeta(h)^{2}}\right)\right) \,.\end{eqnarray*}
Recall that we have already fixed a choice of
$$\tau(h)=h^{3/8}\,.$$ So, for the other parameters,
choosing
$$\zeta(h)=h^{3/16}\,,\qquad\delta=h^{1/8}\,,$$ we
get that the term on the r.h.s. above is of order
$h^{5/8}=o(h^{1/2})$. Integrating with respect to $z\in
\Omega+B(0,\zeta(h))$, we get,
\begin{multline}\label{*****}
\delta\int\sum_{j=1}^N\langle
f_j\,,\,\varphi(x;z)(H_{b,\R^2}-\mathcal
\delta^{-1}W_h)\varphi(x;z)f_j\rangle \,d  z\\
\geq -C|\Omega| h^{5/8}=o(h^{1/2}). \end{multline}
Now we calculate the leading order term. Recall the family
 of projectors $(\Pi_p^L)_{p\geq 1}$  on the Landau
levels introduced in Subsection~\ref{subsec-cs-plane}. We write,
\begin{eqnarray*}
&&\hskip-0.5cm\sum_{j=1}^N\langle
f_j\,,\,\varphi(x;z)H\varphi(x;z)f_j\rangle\\
&&\hskip0.5cm= \sum_{p=1}^\infty
\sum_{j=1}^N\langle
f_j\,,\,\varphi(x;z)H\Pi_p^{L}(h,b)\varphi(x;z)f_j\rangle\\
 &&\hskip0.5cm=\sum_{p=1}^\infty\sum_{j=1}^N\langle
f_j\,,\,\varphi(x;z)(2(p-1) bh-a
h^{3/2})\Pi_p^{L}(h,b)\varphi(x;z)f_j\rangle\\
&&\hskip0.5cm\geq-[a]_+h^{3/2}\sum_{j=1}^N\langle
f_j\,,\,\varphi(x;z)\Pi_1^{L}(h,b)\varphi(x;z)f_j\rangle\,.
\end{eqnarray*}
The operator $ \varphi(x;z)\Pi_1^{L}(h,b)\varphi(x;z)$ is
positive. Therefore,
$$\sum_{j=1}^N\langle
f_j\,,\,\varphi(x;z)\Pi_1^{L}(h,b)\varphi(x;z)f_j\rangle\leq
{\rm
tr}\left(\varphi(x;z)\Pi_1^{L}(h,b)\varphi(x;z)\right)\,.$$
Using \eqref{eq:3},
$${\rm
tr}\left(\varphi(\cdot;z)\Pi_1^{L}(h,b)\varphi(\cdot;z)\right) =
\frac{b}{2\pi h} \int_\Omega \varphi^2(x;z)\,dx \leq 1\,.$$ Hence,
upon integrating w.r.t. $z\in\Omega+B(0,\zeta(h))$, we conclude the
following lower bound,
\begin{align}\label{6*}
(1-\delta)\sum_{j=1}^N\langle f_j\,,\,\varphi(x;z)H
\varphi(x;z)f_j\rangle&\geq
-\frac{h^{1/2}b[a]_+}{2\pi}(1-\delta)\int_{\R^2}\int_\Omega\varphi^2(x;z)\,dx\,dz\nonumber\\
&\geq-\frac{h^{1/2}b|\Omega|}{2\pi}[a]_+-\mathcal O(h^{5/8})\,.
\end{align}
Combining (\ref{*})-(\ref{****}), (\ref{*****}) and (\ref{6*}), we
get the following lower bound,
$$\sum_{j=1}^N\langle f_j\,,\,H
f_j\rangle\geq-\frac{h^{1/2}b|\Omega|}{2\pi}[a]_+-\mathcal
O(h^{5/8})\,,
$$
 which is what we desire to prove.
\subsection{Upper bound}
We just construct a trial density  $\gamma=\gamma_{\rm
int}+\gamma_{\rm bnd}$ and estimate ${\rm tr}\,(H\gamma)$. We take
$\gamma_{\rm int}$ to be
\begin{align*}
\gamma_{\rm int}= \psi_2\left(\frac{{\rm dist}(\cdot,\partial\Omega)}{4h^{3/8}}\right)
    \int_{\R^2}
\chi_h(\cdot;z)\Pi_1^L(h,b)\chi_h(\cdot;z)
d z \,\psi_2\left(\frac{{\rm
    dist}(\cdot,\partial\Omega)}{4h^{3/8}}\right)
\,,\end{align*}
and $\gamma_{\rm bnd}$ exactly as given in \eqref{den-ub}.

Notice that $\gamma_{\rm bnd}$ and $\gamma_{\rm int}$ act as direct sums since their integral kernels have disjoint support.

By calculating ${\rm tr}\,(H\gamma)={\rm tr}\,(H\gamma_{\rm
int})+{\rm tr}\,(H\gamma_{\rm bnd})$ we will get the desired upper
bound. The calculation of ${\rm tr}(H\gamma_{\rm bnd})$ has already
been carried out in Section~\ref{sec-thm1-ub}. In order to calculate
${\rm tr}\,(H\gamma_{\rm int})$, we define,
$$
\varphi(x;z):=\psi_2\left(\frac{{\rm
    dist}(x,\partial\Omega)}{4h^{3/8}}\right)\,\chi_h(x;z)\,,\quad
    \gamma(z)=\varphi(\cdot;z)
\Pi_1^L(h,b)\varphi(\cdot;z)\,.$$ Then,
$${\rm tr}\,(H\gamma_{\rm int})=\int{\rm
tr}\,(H\gamma(z))\, d z\,.$$ Since $\Pi_1^L$ is a
projector, it follows that, \begin{eqnarray*} {\rm
tr}\,(H\gamma(z))&=&{\rm tr}\, \left(H\varphi(x;z)
\Pi_1^L(h,b)\varphi(x;z)\right)\\
&=& {\rm tr}\,\left(\Pi_1^L(h,b)\varphi(x;z) H
\varphi(x;z)\Pi_1^L(h,b)\right)\,.\end{eqnarray*} Using the
following localization formula: \begin{multline*} \langle H\chi
f\,,\,\chi f\rangle={\rm Re}\,\langle\chi Hf\,,\chi
f\rangle+h^2\|\,|\nabla\chi|^2f\,\|^2\,,\\
\forall~\chi\in
C_0^\infty(\Omega)\,,~\forall~f\in D(H)\,,\end{multline*} we deduce
that
\begin{multline}\label{ub-int-key}
{\rm tr}\,\left(\Pi_1^L(h,b)\varphi(x;z) H
\varphi(x;z)\Pi_1^L(h,b)\right)=\\{\rm Re}\left[{\rm
tr}\,\left(\Pi_1^L(h,b)\varphi(x;z)^2 H
\Pi_1^L(h,b)\right)\right] +{\rm
tr}\,\left(\Pi_1^L(h,b)V_h(x;z)
\Pi_1^L(h,b)\right)\,,\end{multline} where
$$V_h(x;z)=h^2\left|\nabla\varphi(x;z)\right|^2\,.$$
Using that the trace is cyclic, we get
\begin{eqnarray}
&&\hskip-0.5cm{\rm tr}\,\left(\Pi_1^L(h,b)\varphi(x;z)^2 H
\Pi_1^L(h,b)\right)={\rm tr}\,\left(\varphi(x;z) H
\Pi_1^L(h,b)\varphi(x;z)\right)\,,\label{UB-int-tr1}\\
&&\hskip-0.5cm{\rm tr}\,\left(\Pi_1^L(h,b) V_h(x;z)
\Pi_1^L(h,b)\right)={\rm
tr}\,\left(\sqrt{V_h}\,\Pi_1^L(h,b)\sqrt{V_h}\right)\,.\label{UB-int-tr2}
\end{eqnarray}
Now $H\Pi_1^L(h,b)=-a h^{3/2}\Pi_1^L(h,b)$. Therefore, we get,
\begin{multline*} {\rm
tr}\,\left(\varphi(x;z) \Pi_1^L(h,b)H
\Pi_1^L(h,b)\varphi(x;z)\right)\\
=-a h^{3/2}{\rm
tr}\,\left(\varphi(x;z) \Pi_1^L(h,b)\varphi(x;z)\right)\,.
\end{multline*}
Notice now that, (recall that $\varphi(x;z) \Pi_1^L(h,b)
\varphi(x;z)$ is a kernel operator),
\begin{eqnarray*}
&&\hskip-0.5cm\int_{\R^2}{\rm tr}\,\left(\varphi(x;z)
\Pi_1^L(h,b)\varphi(x;z)\right) d  z\\
&&\hskip1cm=\frac{b}{2\pi h}\int_\Omega\left|\psi_2\left(\frac{{\rm
dist}(x,\partial\Omega)}{4h^{3/8}}\right)\right|^2\, d  x\\
&&\hskip1cm=\frac{b}{2\pi
h}\int_\Omega\left(1-\left|\psi_1\left(\frac{{\rm
dist}(x,\partial\Omega)}{4h^{3/8}}\right)\right|^2\right)\, d  x\\
&&\hskip1cm=\frac{b}{2\pi h}\left(|\Omega|
-\int_\Omega\left|\psi_1\left(\frac{{\rm
dist}(x,\partial\Omega)}{4h^{3/8}}\right)\right|^2\, d  x\right)\,.
\end{eqnarray*}
The function $\psi_1\left(\frac{{\rm
dist}(x,\partial\Omega)}{4h^{3/8}}\right)$ being supported in
$\Omega(4h^{3/8})$, its $L^2$ integral becomes small of the order
$\mathcal O(h^{3/8})$. Therefore, coming back to (\ref{UB-int-tr1}),
we get finally, \begin{equation}\label{ub-int-tr1=} \int_{\R^2}{\rm
tr}\,\left(\varphi(x;z) \Pi_1^L(h,b)H
\Pi_1^L(h,b)\varphi(x;z)\right)\, d z=-\frac{a
b|\Omega|}{2\pi } h^{1/2} +\mathcal O(h^{5/8})\,.\end{equation} We
need next to estimate the trace \eqref{UB-int-tr2}. Actually,
\begin{eqnarray*}
&&\hskip-1cm{\rm
tr}\,\left(\sqrt{V_h}\,\Pi_1^L(h,b)\sqrt{V_h}\right)\\
&&\hskip0.5cm=\frac{b}{2\pi
h}\int_\Omega V_h(x)\, d  x\\
&&\hskip0.5cm\leq\frac{bh}{\pi}\int_\Omega\left(|\nabla
\chi_h|^2+\left|\nabla\psi_{2}\left(\frac{{\rm
dist}(x,\partial\Omega)}{4h^{3/8}}\right)\right|^2\right)\, d  x\\
&&\hskip0.5cm\leq
C\frac{bh}{\pi}\left(\zeta(h)^{-2}+h^{-3/8}\right)\,.
\end{eqnarray*}
Choosing $\zeta(h)=h^{3/16}$ then coming back to (\ref{UB-int-tr2}),
we get, \begin{equation}\label{ub-int-tr2=}{\rm
tr}\,\left(\Pi_1^L(h,b) V_h(x;z) \Pi_1^L(h,b)\right)=\mathcal
O(h^{5/8})\,.
\end{equation}
We integrate (\ref{ub-int-key}) w.r.t. $z\in
\Omega+B(0,\zeta(h))$ and we substitute (\ref{ub-int-tr1=}) and
(\ref{ub-int-tr2=}) in the resulting formula to get,
\begin{multline*}
{\rm tr} (H\gamma_{\rm int})=\int{\rm
tr}\,\left(\Pi_1^L(h,b)\varphi(x;z) H
\varphi(x;z)\Pi_1^L(h,b)\right) d  z \\
=-\frac{a b|\Omega|}{2\pi } h^{1/2} +\mathcal
O(h^{5/8})\,.\end{multline*}

\section*{Acknowledgements}
The authors were supported by a Starting Independent Researcher
  grant by the ERC under the FP7. SF is also supported by the Danish
Research Council and the Lundbeck Foundation.

\end{document}